\documentclass[reqno,oneside]{amsart}

\usepackage[utf8]{inputenc}
\usepackage[super]{nth}
\usepackage{xcolor}
\usepackage{comment,fancyhdr,tikz,amsmath,amssymb,amsthm,mathtools,centernot,caption,float,graphicx,makecell,array,float,verbatim,parallel,geometry,ragged2e,xcolor,bbm,mathrsfs,pgfkeys,mathabx}
\newcommand\join\vee
\newcommand\meet\wedge
    \usepackage[shortlabels]{enumitem}

\usepackage[export]{adjustbox}
\usepackage[
    colorlinks=false,
    citebordercolor=green,
    linkbordercolor=red
    ]{hyperref}

    \usepackage[capitalise]{cleveref}
    \crefformat{equation}{(#2#1#3)}
    \crefformat{enumi}{(#2#1#3)}
    \crefformat{section}{\S#2#1#3}
    \crefformat{subsection}{\S#2#1#3}
    \crefformat{subsubsection}{\S#2#1#3}
    \numberwithin{equation}{section}
    \newcommand{\crefdefpart}[2]{%
        \hyperref[#2]{\namecref{#1}~\labelcref*{#1}~\ref*{#2}}}
    
    \let\etoolboxforlistloop\forlistloop % save the good meaning of \forlistloop
    \usepackage{autonum}
    \let\forlistloop\etoolboxforlistloop % restore the good meaning of \forlistloop
    \makeatletter
    \newcommand{\blx@noerroretextools}{}
    \makeatother
    
    \usepackage[
        backend=biber,
        style=alphabetic,
        maxbibnames=50,
        maxalphanames=50,
        maxcitenames=50
        ]{biblatex}
    \newcommand{\Z}{\mathbb{Z}}
    
    \newcommand{\R}{\mathbb{R}}
    \newcommand{\C}{\mathbb{C}}
    
    \def\sph^#1{\mathbb S^{#1}}
    \newcommand\p[1]{\left(#1\right)}
    \newcommand\abs[1]{\left|#1\right|}
    \renewcommand\b[1]{\left[#1\right]}

    \newcommand\jp[1]{\left\langle#1\right\rangle}
    
    \newcommand{\spn}{\text{span}}

    \newcommand\eps\varepsilon
    \newcommand\pji\varphi
    \newcommand\sig\varsigma

    \newcommand\beq{\begin{equation}}
    \newcommand\eeq{\end{equation}}
    \newcommand\inv{^{-1}{}}

    \def\XXint#1#2#3{{\setbox0=\hbox{$#1{#2#3}{\int}$ }
    \vcenter{\hbox{$#2#3$ }}\kern-.6\wd0}}

    \newtheoremstyle{thmst}% Theorem
        {5pt}% Space above
        {5pt}% Space below
        {}% Body font 
        {}% Indent amount
        {\bfseries}% ⟨Theorem head font⟩
        {.}% ⟨Punctuation after theorem head ⟩
        {.5em}% Space after theorem head 
        {}%
    \newtheoremstyle{rmkst}% Remark
        {5pt}% Space above
        {5pt}% Space below
        {}% Body font 
        {}% Indent amount
        {\bfseries}% ⟨Theorem head font⟩
        {.}% ⟨Punctuation after theorem head ⟩
        {.5em}% Space after theorem head 
        {}%
    \newtheoremstyle{lmst}% Lemma
        {5pt}% Space above
        {5pt}% Space below
        {}% Body font 
        {}% Indent amount
        {\bfseries}% ⟨Theorem head font⟩
        {.}% ⟨Punctuation after theorem head ⟩
        {5pt plus 1pt minus 1pt}% Space after theorem head 
        {}%
    \newtheoremstyle{prpst}% Proposition
        {5pt}% Space above
        {5pt}% Space below
        {}% Body font 
        {}% Indent amount
        {\bfseries}% ⟨Theorem head font⟩
        {.}% ⟨Punctuation after theorem head ⟩
        {.5em}% Space after theorem head 
        {}%
    \newtheoremstyle{defst}% Proposition
        {5pt}% Space above
        {5pt}% Space below
        {}% Body font 
        {}% Indent amount
        {\bfseries}% ⟨Theorem head font⟩
        {.}% ⟨Punctuation after theorem head ⟩
        {.5em}% Space after theorem head 
        {}%
    \newtheoremstyle{wlst}% White Lie
        {5pt}% Space above
        {5pt}% Space below
        {}% Body font 
        {}% Indent amount
        {\bfseries}% ⟨Theorem head font⟩
        {.}% ⟨Punctuation after theorem head ⟩
        {.5em}% Space after theorem head 
        {}%
    \newtheoremstyle{asmpst}% White Lie
        {5pt}% Space above
        {5pt}% Space below
        {}% Body font 
        {}% Indent amount
        {\bfseries}% ⟨Theorem head font⟩
        {.}% ⟨Punctuation after theorem head ⟩
        {.5em}% Space after theorem head 
        {}%
        
    \theoremstyle{thmst}
        \newtheorem{theorem}{Theorem}[section]
        \newtheorem{conjecture}[theorem]{Conjecture}
        
    \theoremstyle{rmkst}
        \newtheorem{remark}[theorem]{Remark}
    \theoremstyle{defst}
        
    \theoremstyle{defst}
        
    \theoremstyle{lmst}
        \newtheorem{lemma}[theorem]{Lemma}
    \theoremstyle{prpst}
        \newtheorem{proposition}[theorem]{Proposition}
    \theoremstyle{wlst}
        
    \theoremstyle{asmpst}
        
    \theoremstyle{thmst}
        \newtheorem{corollary}[theorem]{Corollary}

        \newlist{defenum}{enumerate}{1} % should only occur inside definition env.
        \setlist[defenum]{label=(\roman*),ref=\thedefinition\,(\roman*)}
        \crefname{defenumi}{Definition}{Definitions}
        
        \newlist{lemenum}{enumerate}{1} % should only occur inside definition env.
        \setlist[lemenum]{label=(\roman*),ref=\thelemma\,(\roman*)}
        \crefname{lemenumi}{Lemma}{Lemmas}
        
        \newlist{corenum}{enumerate}{1} % should only occur inside definition env.
        \setlist[corenum]{label=(\roman*),ref=\thecorollary\,(\roman*)}
        \crefname{corenumi}{Corollary}{Corollaries}
        
        \newlist{thmenum}{enumerate}{1} % should only occur inside definition env.
        \setlist[thmenum]{label=(\roman*),ref=\thetheorem\,(\roman*)}
        \crefname{thmenumi}{Theorem}{Theorems}
        
        \newlist{rmkenum}{enumerate}{1} % should only occur inside definition env.
        \setlist[rmkenum]{label=(\roman*),ref=\theremark\,(\roman*)}
        \crefname{rmkenumi}{Remark}{Remarks}
        
        \newlist{prpenum}{enumerate}{1} % should only occur inside definition env.
        \setlist[prpenum]{label=(\roman*),ref=\theproposition\,(\roman*)}
        \crefname{prpenumi}{Proposition}{Propositions}
        
        \newlist{axenum}{enumerate}{1} % should only occur inside definition env.
        \setlist[axenum]{label=(\roman*),ref=\theaxiom\,(\roman*)}
        \crefname{axenumi}{Axiom}{Axioms}
        
        \newlist{defcrit}{enumerate}{1} % should only occur inside definition env.
        \setlist[defcrit]{label=(\roman*),ref=\thedefinition\,(\roman*)}
        \crefname{defcriti}{Definition}{Definitions}
        
        \newlist{lemcrit}{enumerate}{1} % should only occur inside definition env.
        \setlist[lemcrit]{label=(\alph*),ref=\thelemma\,(\alph*)}
        \crefname{lemcriti}{Lemma}{Lemmas}
        
        \newlist{thmcrit}{enumerate}{1} % should only occur inside definition env.
        \setlist[thmcrit]{label=(\alph*),ref=\thetheorem\,(\alph*)}
        \crefname{thmcriti}{theorem}{theorems}
        
        \newlist{rmkcrit}{enumerate}{1} % should only occur inside definition env.
        \setlist[rmkcrit]{label=(\alph*),ref=\theremark\,(\alph*)}
        \crefname{rmkcriti}{remark}{remarks}
        
        \newlist{prpcrit}{enumerate}{1} % should only occur inside definition env.
        \setlist[prpcrit]{label=(\alph*),ref=\theproposition\,(\alph*)}
        \crefname{prpcriti}{proposition}{propositions}

    \newcommand\ext{\mathcal E}

    \newcommand\dd[2]{\frac{\text d}{\text d #1}\left[#2\right]}

        \makeatletter
        \let\ifnc\@ifnextchar
        \makeatother

            \reversemarginpar
            \def\redit {\marginpar{\raggedleft{See edit $\implies$}}\color{red}}
            \def\rs#1.{\redit #1.\color{black}}
            \def\rsm#1.{{\color{red} #1.}}
            \def\rend{\color{black}}
            \def\bedit {\marginpar{\raggedleft{See edit $\implies$}}\color{blue}}
            \def\bs#1.{\bedit #1.\color{black}}
            \def\bsm#1.{{\color{blue} #1.}}

            \def\ind#1_#2{\left\{#1_{#2}\right\}}
            \def\<#1>{\jp{#1}}
            \def\-#1/{{}_{#1}}
            \makeatletter
                \newcommand\pl@write[3]{%
                    \left\|#1\right\|_{L^{#2}#3}%
                    }
                \def\pl #1__#2{%
                    \def\pl@arg@i{#1}%
                    \def\pl@arg@ii{#2}%
                    \def\pl@arg@iii{\alpha}%
                    \futurelet\next\pl@eval%
                    }
                \def\pl@eval{%
                    \ifx\next\bgroup%
                            \expandafter\pl@eval@iii%
                        \else%
                            \expandafter\pl@eval@ii%
                        \fi%
                    }
                \def\pl@eval@iii#1{%
                    \pl@write\pl@arg@i\pl@arg@ii{\p{#1}}%
                    }
                \def\pl@eval@ii{%
                    \pl@write\pl@arg@i\pl@arg@ii{}%
                    }
            \makeatother
            \makeatletter
                \newcommand\ef@write[3]{%
                    ^{#1\frac{#2}{#3}}{}%
                    }
                \newcommand\Ef@write[3]{%
                    ^{#1#2/#3}{}%
                    }
                \def\ef #1{%
                    \def\ef@arg@i{}%
                    \def\ef@arg@ii{#1}%
                    \futurelet\next\ef@eval%
                    }
                \def\ief #1{%
                    \def\ef@arg@i{-}%
                    \def\ef@arg@ii{#1}%
                    \futurelet\next\ef@eval%
                    }
                \def\ef@eval{%
                    \ifx\next/%
                            \expandafter\ef@eval@iii%
                        \else%
                            \expandafter\ef@eval@ii%
                        \fi%
                    }
                \def\ef@eval@iii/{%
                    \expandafter\ef@eval@v%
                    }
                \def\ef@eval@v#1{%
                    \Ef@write\ef@arg@i\ef@arg@ii{#1}}
                \def\ef@eval@ii{%
                    \expandafter\ef@eval@iv%
                    }
                \def\ef@eval@iv#1{%
                    \ef@write\ef@arg@i\ef@arg@ii{#1}%
                    }
            \def\e@writep #1{%
                ^{+#1}{}%
                }
            \def\e@writen #1{%
                ^{-#1}{}%
                }
            \def\e #1{%
                \ifx#1-%
                        \expandafter\e@writen%
                    \else%
                        \ifx#1+%
                                \expandafter\e@writep%
                            \else%
                                ^{#1}{}%
                            \fi%
                    \fi%
                }
            \def\e@2 {%
                }
            \newcommand\ie[1]{^{-#1}{}}
            \newcommand\eo[1]{^{#1-1}{}}
            \newcommand\ieo[1]{^{-#1+1}{}}
            \makeatother
            \makeatletter
            \newcommand{\undersim}[1]{\mathrel{\mathpalette\@undersim{#1}}}
            \newcommand{\@undersim}[2]{%
              \vcenter{%
                \ialign{%
                  ##\cr
                  $\m@th#1#2$\cr
                  \noalign{\nointerlineskip\kern.2ex}
                  $\m@th#1\sim$\cr
                  \noalign{\kern-.4ex}
                }%
              }%
            }
            \makeatother
            \def\subsetsim{\undersim{\subset}}

\title{Power loss for the {Mizohata-Takeuchi} Conjecture on {$C^k$} convex hypersurfaces}
\author{Hannah Cairo and Ruixiang Zhang}
\date{December, 2025}
\def\bb{\mathfrak b}
\def\dd{\phi}
\begin{document}
\begin{abstract}
     We find a family of compact $C^k$ hypersurfaces where the local Mizohata-Takeuchi Conjecture fails with a power loss of $R^{\frac{n-1}{n-1+k}-\eps}$. Moreover, this family is dense in the $C^k$ topology, and so the local Mizohata-Takeuchi conjecture fails for many convex hypersurfaces. In particular, the local Mizohata-Takeuchi Conjecture fails with a power loss of $R^{\frac{n-1}{n+1}-\eps}$ for many $C^2$ convex hypersurfaces. This power  matches the best known upper bound in \cite{ciw-mt-24} up to the endpoint. For the proof, our weight is positive definite like in \cite{cairo-counterexample-25}, and our construction is based on a projection of a higher rank lattice. As a by-product, we also construct compact convex $C^2$ hypersurfaces whose rescaling contains many lattice points in any dimension. %We show that counterexamples to the local Mizohata-Takeuchi Conjecture in $\R^n$ with a power loss of $R^{\frac{n-1}{n-1+k}-\eps}$ are dense in the $C^k$ topology of all hypersurfaces. 
\end{abstract}
\maketitle
\section{Introduction}\label{introsec}

\subsection{Main results}

This paper concerns the Mizohata-Takeuchi Conjecture, a weighted $L^2$ estimate for free Schr\"{o}dinger solutions and alike:

\begin{conjecture}[Mizohata-Takeuchi]\label{originalMT}
        For any compact $C^2$ hypersurface $\Sigma$ with hypersurface measure $d\sigma$, if we define the \emph{Fourier extension operator} $\ext_{\Sigma}: L^1 (\Sigma, d\sigma) \to L_{loc}^1 (\R^n)$ to be \[\ext_{\Sigma} f (x) = (fd\sigma)^{\check{}}(x),\] then for every non-negative weight $w(x)$ on $\R^n$,     
        \[\pl\ext_{\Sigma} f__2{\R^n;w(x)dx)}^2\lesssim \|\mathcal M w\|_{L^{\infty} (S^{n-1})} \cdot \pl f__2{\Sigma}^2,
            \label{original-mt-inequality}\]
            where $\mathcal M$ denotes the Kakeya-type maximal operator
    \[\mathcal M f(\nu):=\sup_{\text{unit tube }T\|\nu}\int_Tf,\text{ for } \nu \in S^{n-1}.
        \label{kakeya-type-maximal}\]
\end{conjecture}

In this paper, we will often fix $\Sigma$ and abbreviate $\ext_{\Sigma}$ and $L^2 (\Sigma, d\sigma)$ as $\ext$ and $L^2 (\Sigma)$. The right-hand side of \cref{original-mt-inequality} is the first ``natural'' guess to bound the left-hand side, since it is familiar that $\ext f$ can be decomposed into wave packets that propagates along straight lines. This conjecture originated in the study of dispersive PDE with potentials \cite{takeuchi-necessary-conditions-74, takeuchi-cauchy-80, mizohata-cauchy-85} by Takeuchi and Mizohata. The name ``Mizohata-Takeuchi'' was popularized e.g. by the work \cite{barcelo1997weighted}.

\cref{originalMT} is a typical problem in Fourier restriction theory, where one is interested in the quantitative behavior of $\ext f$. On the other hand, it is also quite unusual. Indeed the role of curvature is not so clear in this problem: It follows from Plancherel that \cref{originalMT} holds when $\Sigma$ is a hyperplane, and there one only needs to restrict $T$'s on the right-hand side to those ones orthogonal to $\Sigma$.  For this reason, sometimes this conjecture is formulated in other slightly different but  ``natural'' ways, for example only for convex $C^2$ surfaces (where one has more tools in Fourier restriction theory and it may seem the conjecture is more likely to hold), or only for $T$'s parallel to some normal vector of $\Sigma$, see e.g. Conjecture 1.9 in \cite{carbery2023disentanglement}. For more history, variants and connections of this conjecture, see for example the talk \cite{Carbery-survey-19} or the recent thesis \cite{ferrante2024different}. 

Even in dimension $2$, \cref{originalMT} is out of reach by classical methods. This is perhaps for deeper reasons -- Next we discuss two more surprising aspects of the conjecture revealed by recent studies:

%Despite the classical appearance of \cref{originalMT}, it has two unusual aspects. 

First, in the talk \cite{guth-barrier-22}, Guth explained that the usual ``wave packet decomposition axioms'', which have been very useful in proving estimates like decoupling or refined Strichartz, fail to prove \cref{originalMT}: There must be a ``power-loss'' when one tries to prove the conjecture only using these axioms. i.e. one loses a quantity of the form $R^C$ on the right-hand side of \cref{original-mt-inequality}, if $w$ is assumed to have support in $B_R$ (an $R$-ball) with $R>1$. Later \cite{ciw-mt-24} fleshed out the details in Guth's outline and determined there needs to be a power loss of $R^{\frac{n-1}{n+1}}$ in dimension $n$ ($n=2$ case was covered in \cite{guth-barrier-22}). A further generalization of this phenomenon to an MT-like conjecture for moment curves and alike was recently done in \cite{clpy-curves-25}.

Second, the first author proved in \cite{cairo-counterexample-25} that \cref{originalMT} fails for \emph{every} hypersurface that is not part of a hyperplane, and examples of $w$ exist where supp$w \subset B_R$ and \cref{original-mt-inequality} fails by a $\log R$-loss. The construction of $w$ is a positive definite function which reduces \cref{original-mt-inequality} to a combinatorial estimate. The counterexample is then constructed by using an embedded $c\log R$-dimensional hypercube. 

In light the above, it became particularly interesting to ask whether \cref{originalMT} can have power loss or not. This is often known as the \emph{Local} Mizohata-Takeuchi Conjecture:

\begin{conjecture}[Local Mizohata-Takeuchi]\label{local-mt}
    For any $C^2$ hypersurface $\Sigma$, the estimate \[\pl\ext_{\Sigma} f__2{B_R;w(x)dx)}^2\lesssim R^\alpha\|\mathcal M w\|_{L^{\infty} (S^{n-1})}\cdot \pl f__2{\Sigma}^2
            \label{eqn-mt-power-loss}\] holds  with arbitrarily small $\alpha>0$.
\end{conjecture}

The aforementioned \cite{guth-barrier-22, ciw-mt-24} show that one can take $\alpha > \frac{n-1}{n+1}$, but cannot go further by the wave packet method.  We also mention a recent work \cite{mulherkar2025random}, where it was proved that a generic class of weights satisfy    \cref{local-mt}.

In this paper, we show that \cref{local-mt} fails by a power of $R^{\frac{n-1}{n-1+k}-\varepsilon}$ for a set of $C^k$ hypersurfaces that are dense in the $C^k$ topology, therefore including many strictly convex ones. More precisely, we prove

\begin{theorem}[Power blowup for Mizohata-Takeuchi for $C^k$ hypersurfaces]\label{powerblowupthm}
    Let $\Sigma\subset\R^n$ be a compact $C^k$ hypersurface. Then there exists an arbitrarily small $C^k$ perturbation $\Sigma'$ of $\Sigma$ such that given any $R>1$ and $\eps>0$, one can find some $f\in L^2(\Sigma')$, some weight $w:\R^n\to[0,1]$ such that
        \[\int_{B_R}\abs{\ext_{\Sigma'}f(x)}^2w(x)dx\gtrsim_\eps R^{\frac{n-1}{n-1+k}-\eps}\pl f__2{\Sigma'}^2\pl\mathcal Mw__\infty{\mathbb S\eo n}.\]
   %Furthermore, $\Sigma'$ may be taken to be $C^\infty$ everywhere except at one point.
\end{theorem}

\cref{powerblowupthm} shows the Local Mizohata-Takeuchi Conjecture fails for many $C^k$ hypersurfaces, including many strictly convex ones. Moreover, in the $C^2$ case with $k=2$, recall \cite{guth-barrier-22, ciw-mt-24} proved the $R^{\frac{n-1}{n+1}+\eps}$ power loss version for every strictly convex $C^2$ hypersurface. \cref{powerblowupthm}  shows that their power-loss result is sharp for general  strictly convex $C^2$ hypersurfaces, up to the endpoint.

In \cref{powerblowupthm}, the way how the blowup depends on smoothness is  unusual in Fourier restriction theory and somewhat surprising. It is also quite intriguing to us that \cref{powerblowupthm} does not say anything about Mizohata-Takeuchi for some of the most familiar hypersurfaces like the paraboloid or the sphere. It remains a mysterious and very challenging question to prove or disprove the Local MT Conjecture for these individual hypersurfaces.

\subsection{A sketch of the  method of our construction}

The starting point of our construction is an observation by the first author \cite{cairo-counterexample-25} that  for positive definite weights $w$, \cref{original-mt-inequality} boils down to a combinatorial problem. Based on this observation, the first author \cite{cairo-counterexample-25} was able to found weights leading to $\log$-blowup based on a projection of a hypercube in very high dimension ($\sim \log R$). In the present paper, we use the projection of a lattice in a high constant dimension (that we call $N$) to construct the counterexample for power blowup.

Along the way of our proof we use probabilistic methods and first prove a few lemmas about intersections between convex bodies and lattice points in \cref{subsec-shape}. We feel they may be of independent interest. Because of these probabilistic considerations, our results only apply to artificial $C^k$ hypersurfaces and our method does not produce counterexamples that lead to power loss for MT for  individual hypersurfaces $\Sigma$ of interest, like the sphere or the paraboloid.

In \cref{Latticediscussionsec}, we will further outline our main idea in a quantitative way using a rank $2$ lattice in $\R^2$. This will fail to construct a counterexample, but the readers will see how it naturally gives conceptually nice perspectives that motivate the higher rank construction.

\subsection{Mizohata-Takeuchi-type estimates for codimension $\geq2$.}\label{subsec-codim}

It is also natural to consider analogs of \cref{originalMT} where $\Sigma\subset\R^n$ is now a submanifold of dimension $m< n-1$. This was considered in \cite{clpy-curves-25} and we generalize Mizohata-Takeuchi by considering the $l$-plane-transforms

\[M_{l} w(b):=\sup_{T: 1-\text{neighborhood of some translate of }b}\int_{T}w dx\]
where $w\geq 0$ and $b \in \text{Gr} (l, n)$ is the Grassmannian of $l$-dimensional subspaces. With  this setup, we are interested in local estimates of the form
\begin{equation}\label{mt-powers-alpha-codimension}
    \pl\ext_{\Sigma} f__2{B_R;w(x)dx)}^2\lesssim R^{n-m-l+ \alpha}\pl\mathcal M_{l}w__\infty{\text{Gr} (l, n)}\pl f__2{\Sigma}^2
\end{equation} 
for some $\alpha$ and wonder whether $\alpha$ can be $0$. Our choice of the exponent is motivated the standard observation that if we take $f = 1_{\Sigma}$ and $w = 1_{B_R}$, we can explicitly compute both sides to show that it is necessary to have $\alpha \geq 0$ in this normalization.

\cite{clpy-curves-25} focused on the $m=1$ case and when the curve is $C^{n+1}$ and well-curved (i.e. having Wronskian away from $0$ all the time). Theorem 1.3 in \cite{clpy-curves-25} asserts that in this case, if one takes $l=1$, then any $\alpha>R\ef2{n(n+1)}$ suffices. In the present paper, we will show in \cref{thm-counterexample} that if we want \cref{mt-powers-alpha-codimension} for all $C^k$ curves with $l=1$, it is necessary to take $\alpha\geq\frac1{1+(n-1)k}$.  Comparing it with Theorem 1.3 of \cite{clpy-curves-25}, we see, interestingly, that when smoothness improves from $C^k (k<\frac{n+2}2)$ to $C^{n+1}$, the answer to the $m=1, l=1$ case does change in all higher dimensions $>2$. %In particular, if $k<\frac{n+2}2$, then the minimal $\alpha$ in the $C^{n+1}$ case is strictly less than the minimal $\alpha$ in the $C^k$ case. Note that $k<\frac{n+2}2$ would require $k<2$ when $n=2$, so it appears that this phenomenon, as it currently stands, manifests only when $n\geq 3$.

\subsection{Notations}  We will often work with a large parameter $R>1$ or a small parameter $\delta \in (0, 1)$. In this context we say non-negative quantities $A\gtrapprox B$ if and only if $A \gtrsim_{\varepsilon} R^{-\varepsilon} B, \forall \varepsilon$ or $A \gtrsim_{\varepsilon} \delta^{\varepsilon} B, \forall \varepsilon$. Similarly we have the notions of $\lessapprox$ and $\approx$. In this paper, our main task is to prove the main Theorem \ref{powerblowupthm} and a generalization (\cref{thm-counterexample}), and we will have a large constant dimension $N$ that only depends on the $\varepsilon$ in these theorems. We often have constants depending on that $N$ and note that we can hide those constants under $\gtrapprox$, etc.  $A \gg B$ means that, if $B$ is fixed, then $A$ is sufficiently large. 

We use $B_R$ or $B_R^d$ (in case we want to specify the ambient dimension) to define an $R$-ball in $\R^d$. 

We will use two types of smooth cutoffs. For any $n$, let $\bb_1:\R^n\to\C$ be a nonnegative, smooth, radially symmetric, compactly supported bump function with nonnegative Fourier transform and
    \[1_{B_1}\leq\bb_1\leq1_{B_2}\]
By $\bb_R$, we mean the rescaled version $\bb_R(x)=\bb_1(R\inv x)$. We also set $\dd_1=\hat\bb_1$, and $\dd_{R\inv}=\widehat{\bb_R}$. This way, multiplication by $\bb_R$ can be thought of as a smooth localization to $B_R$, and $\dd_{R\inv}$ can be thought of as a scale-$R\inv$ approximation to $\delta_0$. We also write $\bb_{R,x}=\bb_{R}*\delta_{x}$ and $\dd_{R\inv,\xi}=\dd_{R\inv}*\delta_{\xi}$. Given a discrete set $S$, we define $\delta_S=\sum_{s\in S}\delta_s$.

We will use $c, C, C_{\eps}$ to denote  constants $>0$ that can change from line to line. $c, C$ are absolute (possibly depending on the dimension) and $C_{\eps}$ can depend on $\eps$.

If $U\subset\R^n$ is some set, then we write $\mathcal N_a(U)$ for the $a$-neighborhood of $U$.  If $U$ is a set and $V$ is a convex set, then we write $U\subsetsim V$ to denote $U\subset CV$ for some $C>0$,  where $CV$ is a rescaling of $V$ by $C$ times around any point in $V$.

\section*{Acknowledgements} RZ is supported by NSF DMS-2143989 and the Sloan Research Fellowship. He would like to thank Larry Guth and Hong Wang for helpful discussions, in particular on  whether a Gauss sum example that is related to a rank $2$ lattice can serve as a counterexample for the Mizhohata-Takeuchi conjecture. %The authors would like to thank Anthony Carbery for making valuable  comments, especially  on  the background and preliminary reductions of Stein's conjecture that led to the discussion in \cref{Steinsconjsec}.

\section{An initial reduction}
    As a starting point, we present an equivalent formulation of the local Mizohata-Takeuchi Conjecture. Under the hypotheses in \cref{subsec-codim}, let $C_{\Sigma,l}(R)$ be the smallest constant such that
    \[\pl\ext_{\Sigma} f__2{B_R;w(x)dx)}^2\lesssim R^{n-l-m}C_{\Sigma,l}(R)\pl f__2{\Sigma}^2\pl\mathcal M_lw__\infty{Gr (l, n)}\]
    for all weights $w$ supported in $B_R$. Furthermore, let $C_{\Sigma,l}'(R)$ be the smallest constant such that
    \[\pl\hat w__2{\Sigma}\leq R\ef{n-l-m}2C_{\Sigma,l}'(R)\pl w__1\ef12\pl\mathcal M_lw__\infty{Gr (l, n)}\ef12 \label{equivMT}\]
    for all weights $w$ supported in $B_R$.
    
    \begin{lemma}\label{initialreductionlem}
        $C_{\Sigma,l}(R)\approx\p{C_{\Sigma,l}'(R)}^2$.
    \end{lemma} 

    \begin{remark}
        The implied constant in Lemma \ref{initialreductionlem} is allowed to depend on $\Sigma$, and we will suppress this dependence.
    \end{remark}
    \begin{proof}[Proof of Lemma \ref{initialreductionlem}]
        First, we show $C_{\Sigma,l}(R)\geq \p{C_{\Sigma,l}'(R)}^2$ as follows: Note that
            \[\sup_f\pl f__2{\Sigma}\ie2\int_{B_R}|\ext f(x)|^2w(x)dx\geq\sup_f \pl f__2{\Sigma}\ie2 \|\ext f\cdot w\|_{L^1 (B_R)}^2\pl w__1{\R^n}\inv
                \label{duality-1}\]
        by Cauchy-Schwarz, and
            \[\sup_f\pl f__2{\Sigma}\inv\|\ext f\cdot w\|_{L^1 (B_R)}=\sup_f\pl f__2{\Sigma}\inv\int_\Sigma f(\sig)\hat w(\sig)d\sigma(\sig)=\pl \hat w__2{\Sigma}.\]
        by duality. Therefore,
            \[\pl\hat w__2{\Sigma}^2\pl w__1{\R^n}\inv\leq\sup_{f\in L^2{\Sigma;d\sigma}}\pl f__2{\Sigma}\ie2\int_{B_R}|\ext f(x)|^2w(x)dx\leq R^{n-l-m}C_{\Sigma,l}(R)\pl\mathcal M_lw__\infty{SO(n)},\]
        which implies $C_{\Sigma,l}(R)\geq\p{C_{\Sigma,l}'(R)}^2$.
    
        Now, we would like to show $C_{\Sigma,l}(R)\lessapprox \p{C_{\Sigma,l}'(R)}^2$. Recall from \cref{subsec-codim} that both $C_{\Sigma,l}$ and $C_{\Sigma,l}'$ are $\gtrsim1$. If we assume $\|f\|_{L^2 (\Sigma)}\sim 1$, it then suffices to bound each $\int_{S_{j, e_1, e_2}} |\ext f|^2 w dx$. Here  signatures $e_1, e_2 \in \{+, -\}$ and $S_{j, e_1, e_2}$ is the set where $|\ext f| \sim 2^j \in [R^{-\frac{m}{2}}, C]$ and Re$(f)$ and Im$(f)$ have signs $e_1, e_2$, respectively.\footnote{It is easy to see $\|\ext f\|_{\infty}\lesssim \|f\|_{L^2 (\Sigma)} = 1$. For contributions on the subset $|\ext f| < R^{-\frac{m}{2}}$, we use the wasteful estimate $\pl\mathcal M_lw__\infty\gtrsim R^{-(n-l)}\pl w__1$ and the fact $C_{\Sigma,l}'(R)\gtrsim 1$.} Let $w_{j, e_1, e_2} = 1_{S_{j, e_1, e_2}} w$. We see 
        \begin{align}
            \int_{S_{j, e_1, e_2}} |\ext f|^2 w dx \sim &\p{\int \ext f\cdot w_{j, e_1, e_2} dx}^2 \|w_{j, e_1, e_2}\|_{L^1}^{-1}\sim\bigg(\int_{\Sigma}f\cdot \hat{w}_{j, e_1, e_2}d\sigma\bigg)^2\|w_{j, e_1, e_2}\|_{L^1}^{-1}\\\leq &(C_{\Sigma,l}' (R))^2 \|f\|_{L^2 (\Sigma)}^2 \|\mathcal M_lw_{j, e_1, e_2}\|_{L^{\infty}}
            \end{align}
        by the definition of $C_{\Sigma,l} '$ and Cauchy-Schwarz. The term $\|\mathcal{M}w_{j, e_1, e_2}\|_{L^{\infty}}$ is then bounded by $\|\mathcal{M}w\|_{L^{\infty}}$  and adding this up over all $(j, e_1, e_2)$ will result in an acceptable $\log$ factor.
        \end{proof}

\section{Warm up: A study of the lattice example}\label{Latticediscussionsec}
Before providing the counterexample, we present an example that casts doubt on the Mizohata-Takeuchi Conjecture for certain hypersurfaces $\Sigma$, but nonetheless falls short of a counterexample.

Suppose for simplicity that $n=2$ and $l=1$, then set
    \[h_0=\sum_{\xi_0\in R\ief12\Z^2}\delta_{\xi_0}\]
    \[h=R\inv(h_0\bb_1)*\dd_{R\inv}
        \label{def-h-lattice-example}\]
Thus, $h$ is an $L^1$-normalized approximation to the (truncated) lattice $L_0 : =R\ief12\Z^2\cap B_1$, while $\hat{h}$ is an approximation to the dual lattice $R\ef12\Z^2\cap B_R=:L_0'$, which has $L^{\infty}$-norm $\sim 1$ and is locally constant at scale $1$.

Motivated by \cite{cairo-counterexample-25}, we now set $w(x)=|\hat h(x)|^2$. Recall \cref{initialreductionlem}, the local MT Conjecture would imply that
    \[\pl h*\tilde h__2{\Sigma}\lessapprox \pl h__2\pl\mathcal Mw__\infty\ef12
        \label{lattice-mt}.\]
        
 Note that $h*\tilde h$ behaves roughly the same as $h$. %(technically, we should use different bump functions). 
 This means that we expect\footnote{This can be easily made rigorous for many examples of $\Sigma$ such as the sphere or the paraboloid.}
    \[\pl h*\tilde h__2{\Sigma} \sim \pl h__2{\Sigma}\sim R(|S\cap L_0|R\inv)\ef12\sim  \pl h__2|S\cap L_0|\ef12,\]
where $S$ is the $R\inv$-neighborhood of $\Sigma$. %Note that the $R\ef12$ term comes from the fact that $\dd_{R\inv}$ has different $L^2$ norms on $\R^2$ and $\Sigma$, since $\Sigma$ has dimension $1$.

Recall $\hat h$ is an approximation of the lattice $L_0 '$ that is locally constant at scale $1$. Since $\| \hat h\|_\infty\sim1$, we have 
    \[\pl\mathcal Mw__\infty\approx\sup_{1\times R-\text{tube }T\subset\R^2}|L_0 '\cap T|.\]

 Thus, \cref{lattice-mt} can be written in the form
\[|L_0 \cap S| \lessapprox \sup_T |L_0 ' \cap T|\label{latticecomparison}\]
    %\[C_{\Sigma,l}(R)\gtrapprox\frac{|L_0\cap S|}{\sup_{T\subset\R^2}|L_0'\cap T|}\]
This inequality is suspicious, since it requires an unconditionally favorable comparison between the lattice point counting near the curve $\Sigma$ and the number of dual lattice counting near a line, that is surprisingly simple and may seem suspicious. Nevertheless, we were unable to obtain a contradiction for any $\Sigma$ we know by this construction. For example, if $\Sigma$ is the parabola, then $|S\cap L_0|\sim R\ef14$, $\sup_{T}|T\cap L_0'|\sim R\ef12$ and \eqref{latticecomparison} holds, even with some extra room. %We do not yet a contradiction.

%We will remedy this obstruction in the rest of this paper as follows. Note that the choice of lattice $L_0$ was arbitrary. To construct a counterexample, we would like the lattice $L_0$ to contain many points on $S$, while limiting the tube-occupancy of the ``dual'' lattice $L_0'$. In \cite{cairo-counterexample-25}, this goal was achieved by choosing a very high-dimensional GAP. The primary obstruction to extending this method beyond logarithmic losses is the difficulty in understanding high-dimensional GAP's and how they intersect hypersurfaces. On the other hand, this section has demonstrated the difficulties of using low-dimensional GAP's (namely, the high tube-occupancy). In light of these obstructions, we are led to consider ``medium-dimension" GAP's.

Next we introduce our main conceptual innovation in this paper. Note that in the above argument, the most important ingredients include the fact that $h * \tilde {h}$ is like $h$, and the reduction to \cref{latticecomparison}. In sections below, we will construct another $h$ around a higher rank generalized arithmetic progression, that can also be viewed as a projection of a higher dimensional lattice to $\R^n$. The dimension $N$ of the lattice is going to be chosen as a large constant depend on the $\varepsilon$ we want in \cref{powerblowupthm}. This grants $h * \tilde {h}$ still looks a lot like $h$, while providing more flexibility (in particular, unlike in the lower rank construction, much less points are automatically on a line) that the analog of \cref{latticecomparison} has more chances to fail. We will prove the analog of \cref{latticecomparison} fails for many $C^k$ perturbations of any given $C^k$ hypersurface by a probabilistic argument that shows  the analog of \cref{latticecomparison} can fail a lot for the projection of a generic lattice. To this end, we prove a main \cref{thm-convex-density} on statistics of intersections between convex bodies and lattices in an arbitrary dimension. We will apply this theorem in many ways, and hope this nice-looking theorem can be of independent interest.

For anything like \cref{latticecomparison} to fail, we need to produce many lattice points on $\Sigma$. We will see that this is possible for some $C^k$-perturbation of any $C^k$-curve, but will be harder when $k$ becomes larger. We intuitively introduce how we produce these using a toy example of a $C^2$ perturbation of a $C^2$ curve here. This will be rigorously turned into a  theorem (\cref{sharplatticethm}) in \cref{richsurfacessection} with the help of \cref{thm-convex-density} later. Let $R>1$ be a parameter and  $L$ be the grid $(R^{-\frac{1}{2}}\Z)^2 \subset \R^2$. Suppose $\Sigma$ is any given compact $C^2$ curve (such as the circle or the parabola). Note that $\Sigma$ can be divided into $\sim R\ef13$ segments of length $\sim R\ief13$. Each segment is contained in a box of dimension $\sim R\ief13\times R\ief23$. Thus if we randomly rotate and translate the curve, we can expect each box contains $\approx 1$ lattice point in $L$. But a random rigid motion image of the union of the box is expected to contain $\sim R\ef13$ points in $L$. If we choose such a good rigid motion and perturb the curve locally, we can expect to find a small $C^2$ perturbation of $\Sigma$ in each box and obtain in this way a new curve that contains $\gtrapprox R\ef13$ points in $L$. %By changing the constants, the perturbation can be arbitrarily small in the $C^2$ topology, and thus the resulting curve can still be strictly convex, for example. Curiously, this numerology coincides with the numerology on the Jarn\'{i}k curve (see e.g. \cite{iosevich2001curvature} for an introduction) but seems more flexible: In Jarn\'{i}k's construction, the curve is always a perturbation of a parabola \cite{martin2003limiting}. Also, this method of construction is more versatile to produce points on $C^k$ submanifolds of arbitrary dimension and codimension.

Our construction is inspired by the $\log$-blowup construction of the first author \cite{cairo-counterexample-25}, where the projection of a very high ($\sim \log R$) dimensional hypercube is used. We found that this kind of high-dimensional-projection ideas in \cite{cairo-counterexample-25} and this paper are rarely used in Euclidean harmonic analysis, and hope to see more applications on other problems in the future.

\begin{remark}
    We remark that the rank $n$ lattice example in $\R^n$ and connections to MT was independently noticed and studied by others like Anthony Carbery \cite{Carbery-survey-19}, Larry Guth and Hong Wang (personal communication with the second author) and Xuerui Yang \cite{y-exp-sums-25} (who used this lattice construction to find counterexamples to a weighted restriction estimate). We also point out that to disprove local MT, it is very important to take the rank much larger than $n$ in our method.
\end{remark}

\section{Convex geometry lemmas and Kakeya-type maximal estimates for higher dimensional lattice configurations}\label{subsec-shape}
    %In this section, we would like to prove several estimates for the lattice-based constructions described in the previous section.
    Our counterexample will be based on high-dimensional lattice configurations. Before going into the proof, we develop some useful tools in this section. These will be based on a theorem (Theorem \ref{thm-convex-density}) about how many points a rotated convex body can intersect a lattice, and how often can this happen, in a higher dimension. We feel this theorem (along with the convex geometry \cref{lemma-intersection-estimate} needed to prove it) may be of independent interest, and are not aware of elsewhere it is recorded, so we include a proof.

    In this section, $N$ will be a large (but eventually fixed) dimension, and $d$ will also denote a dimension. We will use them to state our results and all implied constants can depend on  $N$ or $d$. Whenever there is an implied measure on $SO(N)$ or $SO(d)$, we always mean the standard Haar measure.

    \subsection{Some useful lemmas about the intersection of convex sets}
    
    First we prove some lemmas in convex geometry. Let $d>0$ be any dimension and $T\subset\R^d$ be any bounded symmetric convex set. By John's ellipsoid theorem \cite{john2013extremum}, there exists a nonincreasing dyadic tuple $(a_1,\cdots,a_d)$ and a rectangular box $B\supset T$ with side lengths $a_1\geq \cdots \geq a_d$ such that $T$ contains a translation of $cB$. In this case, we say that $T$ is \emph{equivalent to} the dyadic box $B$, and that $(a_1,\cdots,a_d)$ are the dimensions of $T$, denoted $\dim(T)$.% If two symmetric convex bodies $S_1$ and $S_2$ centered at the origin satisfy that $cS_1\subset S_2 \subset CS_1$ for constants $c, C>0$, we say they are \emph{essentially the same}. When this is not the case, we say they are \emph{essentially different}.\footnote{The definition of being essentially the same depends on the choices $c$ and $C$, but other than the implied constants, all results we prove involving this property will be unchanged for different choices of $c$ and $C$.}

    Let $\mathfrak D$ denote the set of all nonincreasing tuples $s=(s_i,\cdots,s_d)$ of dyadic numbers. We define a partial ordering on $\mathfrak D$ by $s\leq t$ iff $s_i\leq t_i$ for all $i\in[d]$. We may define $\lesssim$, $\gtrsim$, and $\sim$ analogously. Given any $(s,t)$, there is a unique infimum and supremum with respect to $\leq$, which we denote $s\meet t$ and $s\join t$, respectively.

    Given any $s\in\mathfrak D$, let $[s]$ the dyadic box $\left[-\frac12s_1,\frac12s_1\right]\times\cdots\times\left[-\frac12s_d,\frac12s_d\right]$ and let $|s|=\prod_{i=1}^ds_i$ be the volume of $[s]$.
    
    \begin{lemma}\label{lemma-dyadic-contain-prob}
        Let $a=(a_1,\cdots,a_d)$ and $b=(b_1,\cdots,b_d)$ be two nonincreasing dyadic tuples with $a\lesssim b$, and let $g\in SO(d)$ be a uniformly distributed random variable. Then
            \[\mathbb P\Big[\big|g[a]\cap[b]\big|\sim|a|\Big]\sim I(a,b)\]
        where
            \[I(a,b):=\prod_{i,j=1}^d\min(a_i\inv b_j,1)\]
    \end{lemma}
    \begin{proof}
        Let $v_i=ga_ie_i$, i.e. $v_i$ denotes the $i$th sidelength of $g[a]$. It suffices to show that
            \[\mathbb P[v_i\in C_d[b],\,\forall i\in[d]]\sim I(a,b)\]
        for some constant $C_d$. We proceed by induction on $d$, and note that the case $d=1$ is trivial. First, note that
            \[\mathbb P[v_1\in C_d[b]]\sim\frac{|B_{a_1}\cap[b]|}{|B_{a_1}|}\sim\prod_{j=1}^d\min(a_1\inv b_j,1)\]
        Note that $v_i\in v_1^\perp$ for $i\geq 2$, so it suffices to find the dimensions of $[b]\cap v_1^\perp$. To see why, note that if $(b_2',\cdots,b_d')=\dim([b]\cap v_1^\perp)$, then by the induction hypothesis, we have
            \[\mathbb P[|g[a]\cap[b]|\sim|a|]\sim\prod_{j=1}^d\min(a_1\inv b_j)\prod_{i,j=2}^d\min(a_i\inv b_j',1)=I(a,b)\]
        Let $J$ be the largest value of $j\geq1$ with the property that $a_1\lesssim b_J$. It suffices to show that $b_j'\sim b_j$ for $j>J$ and $b_j'\gtrsim a_1$ for $2\leq j\leq J$. To that end, note that it suffices to show that the dimensions of $B_{a_1}\cap [b]\cap v_1^\perp$, say $c_2,\cdots,c_d$, satisfy $c_j\sim a_1$ for $j\leq J$ and $c_j\sim b_j$ for $j\geq J+1$. The box $B_{a_1}\cap[b]$ is equivalent to a box $\tilde B$ with dimensions $(a_1,a_1,\cdots,a_1,b_{J+1},\cdots,b_d)$. Since $|v_1|=a_1$, we may in fact assume that $v_1$ is a sidelength of $\tilde B$. This implies that $\dim(\tilde B\cap v_1^\perp)=(a_1,\cdots,a_1,b_{J+1},\cdots,b_d)$, as desired. This completes the proof.
    \end{proof}
    \begin{lemma}\label{lemma-intersection-estimate}
        Let $S,T\subset\R^d$ be two symmetric convex sets centered at $0$ with dimensions $s,t$, respectively. Let $g\in SO(d)$ be a uniformly distributed random variable. For any $c\in\mathfrak D$ with $c\leq s\meet t$, let $\chi_c$ denote the event that $c\lesssim \dim(gS\cap T)$. Then we have
            \begin{align}
                            \mathbb P[\chi_c]\lesssim       &\,\frac{I(c,s)I(c,t)}{I(c,c)},\\
                \mathbb P[\chi_{s\meet t}]\sim  &\,I(s\meet t,s\join t),\\
                \mathbb P[|gS\cap T|\gtrsim\lambda\inv|s\meet t|]\lesssim%
                                                &\,(\lambda\log\lambda)\eo dI(s\meet t,s\join t)
            \end{align}
        for $\lambda\geq 2$.
    \end{lemma}
    \begin{proof}
        It suffices to consider the case when $S=[s]$ and $T=[t]$. First, note that
            \[\mathbb P[\chi_c]\sim\mathbb P\big[|S\cap gT\cap h_0[c]|\sim|c|\text{, for some }h_0\in SO(d)\big]
                \label{eqn-chi-c}\]
        When $\chi_c$ occurs, we let $h_0\in SO(d)$ be any rotation that satisfies $|S\cap gT\cap h_0[c]|\sim|c|$. When $\chi_c$ does not occur, we leave $h_0$ undefined.
        
        Let $h\in SO(d)$ denote a uniformly distributed random variable, independent of $g$ and $h_0$, and let $\chi_c'$ denote the event that $\big|h[c]\cap S\cap gT\big|\sim|c|$. By \cref{lemma-dyadic-contain-prob},
            \[\mathbb P[\chi_c']\sim I(c,s)I(c,t)
                \label{eq-401}\]
        We would like a way of measuring $\chi_c$, but currently, we only have the tools to measure $\chi_c'$. At the moment, it is unclear how to relate $\mathbb P[\chi_c']$ to $\mathbb P[\chi_c]$. To see how, first note that if $\chi_c'$ occurs, then necessarily $\chi_c$ will occur as well. That is,  we have
            \[\mathbb P[\chi_c']=\mathbb P[\chi_c'|\chi_c]\mathbb P[\chi_c]
                \label{eq-402}\]
        It is very difficult to calculate $\mathbb P[\chi_c'|\chi_c]$ exactly, since $\chi_c$ only gives partial information about the shape of $gS\cap T$. However, we only need find a lower bound since our aim is to estimate $\mathbb P[\chi_c]$. To that end, note that if $\chi_c$ occurs, then $\chi_c'$ will occur as long as $g_0$ is sufficiently close to $h$. To quantify this, suppose that $\big|h_0[c]\cap h[c]\big|\sim|c|$. If this holds and $\chi_c$ holds, then $\chi_c'$ will hold. That is,
            \[\mathbb P[\chi_c'|\chi_c]\gtrsim\mathbb P[\big|h_0[c]\cap h[c]\big|\sim|c|\big|\chi_c]
                \label{eq-403}\]
        with $\sim$ when $c=s\meet t$. Since $h$ and $h_0$ are independent, \cref{lemma-dyadic-contain-prob} implies
            \[\mathbb P[\chi_c'|\chi_c]\gtrsim I(c,c)\]
        Combining this with \cref{eq-401} and \cref{eq-402} yields
            \[\mathbb P[\chi_c]\lesssim\frac{I(c,s)I(c,t)}{I(c,c)}\]
        with $\sim$ when $c=s\meet t$. Note that $s\meet t=(\min(s_1,t_1),\cdots,\min(s_d,t_d))$ and $s\join t=(\max(s_1,t_1),\cdots,\max(s_d,t_d))$, and so
            \[\mathbb P[\chi_c]\lesssim\frac{I(c,s\meet t)I(c,s\join t)}{I(c,c)}
                \label{eq-404}\]
        with $\sim$ when $c=s\meet t$. This implies that
            \[\mathbb P[\chi_{s\meet t}]\sim I(s\meet t,s\join t)\]
        We would like to simplify \cref{eq-404}. To that end, note that \cref{lemma-alg} implies
            \[\frac{I(c,s\meet t)I(c,s\join t)}{I(c,c)}\leq\frac{I(c,s\meet t)I(s\meet t,s\join t)}{I(s\meet t,c)}\cdot\frac{|c|}{|s\meet t|}=\p{\frac{|c|}{|s\meet t|}}\ieo dI(s\meet t,s\join t)\]
        To finish the proof, note that, for a fixed $\lambda$, there are $\sim(\log\lambda)\eo d$ possible $c\in\mathfrak D$ with $|c|\sim\lambda\inv|a|$ and $c\leq a$.
    \end{proof}

    \begin{lemma}\label{lemma-alg}
        Let $\alpha,\beta,\gamma,\eta\in\mathfrak D$.
            \begin{lemenum}
                \item\label{lemma-alg-1}%
                    We have
                        \[\frac{I(\alpha,\beta)}{I(\beta,\alpha)}=\p{\frac{|\alpha|}{|\beta|}}\ie d
                            \label{eqn-lemma-alg-1}\]
                \item\label{lemma-alg-2}%
                    If $\alpha\join\gamma\leq\beta\meet\eta$, then
                        \[\frac{I(\alpha,\eta)I(\beta,\gamma)}{I(\alpha,\gamma)I(\beta,\eta)}\leq\frac{|\alpha\join\gamma|}{|\beta\meet\eta|}
                            \label{eqn-lemma-alg-2}\]
            \end{lemenum}
    \end{lemma}
    \begin{proof}
        For (i), write
            \[\frac{I(\alpha,\beta)}{I(\beta,\alpha)}=\prod_{i,j=1}^d\frac{\min(\alpha_i\beta_j\inv,1)}{\min(\alpha_i\inv\beta_j,1)}=\prod_{i,j=1}^d\alpha_i\beta_j\inv\]
        For (ii), note that since $\gamma\leq \eta$,
            \[  \frac{I(\xi,\eta)}{I(\xi,\gamma)}
                =
                \prod_{i,j=1}^d
                    \begin{cases}
                        1,                      &%
                                \xi_i\leq\gamma_j\\
                        \xi_i\inv\gamma_j,   &%
                                \gamma_j\leq\xi_i\leq\eta_j\\
                        \eta_j\inv\gamma_j,   &%
                                \xi_i\geq\eta_j
                        \end{cases}
                \]
        is a decreasing function of $\xi$, so when calculating the cross-ratio \cref{eqn-lemma-alg-2}, we may neglect the terms with $i\neq j$ and obtain
            \[  \frac       {I(\alpha,\eta)%
                                            I(\beta,\gamma)}%
                            {I(\alpha,\gamma)%
                                            I(\beta,\eta)}%
                \leq
                \prod_{i=1}^d
                    \frac   {%
                                            \beta_i\inv\gamma_i}%
                            {\min(1,\alpha_i\inv\gamma_i)%
                                            \min(1,\beta_i\inv\eta_i)}
                =
                \frac       {|\alpha\join\gamma|}%
                            {|\beta\meet\eta|}.
                \]
    \end{proof}

    %\begin{remark}\label{rmkthinbox}
     %   By taking a limit or directly repeating the proof, \cref{lemma-intersection-estimate} also applies to the case where $S$ or $T$ has some dimensions equal to $0$, i.e. when $S$ or $T$ is a lower dimensional object.
    %\end{remark}

    \subsection{Probabilistic incidence estimates between convex sets and the integer lattice}
    
    \begin{theorem}\label{thm-convex-density}
        Fix $N>0$, and let $g\in SO(N)$ be a uniformly distributed random variable. We have
            \[\mathbb P[|gT\cap\Z^N|\geq K]\lesssim K\ie N\]
        uniformly for every convex subset $T$ with $|T|\lesssim 1$.
    \end{theorem}
\begin{comment}
\begin{remark}
    By carefully keeping track of the rotations of the $d_1\times \cdots \times d_M$-box that contain a fundamental domain of $\mathcal P$, it may be possible to get rid of the $r^{\eps}$ in the proof with more effort. We do not pursue this delicacy here.
\end{remark}
\end{comment}
\begin{proof}
        Assume that $T$ has dimensions $(t_1,\cdots,t_N)$, and let $K_g=|gT\cap\Z^N|$ be an integer-valued random variable. Consider the sublattice $\mathcal P_g\subset\Z^N$ generated by $|gT\cap\Z^N|$, and suppose that $\dim(\mathcal P_g)=M$. Note that the convex hull of $gT\cap\Z^N$ lies in the $M$-dimensional vector space $\spn_\R(\mathcal P_g)$ and has $M$-volume $\leq|gT\cap\spn_\R(\mathcal P_g)|$, so it can be partitioned into $\geq K_g-M$ many $M$-simplices, each of volume $\geq(M!)\inv\text{cov}(\mathcal P_g)$, where cov$(\mathcal P_g)$ denotes the covolume of $\mathcal P_g$. %, which is the determinant of a $\Z$-basis of $\mathcal P_g$. 
        Therefore, we have
            \[|gT\cap\spn_\R(\mathcal P_g)|\geq\frac{K_g-M}{M!}\text{cov}(\mathcal P_g)\gtrsim K_g\cdot \text{cov}(\mathcal P_g)\]\def\cov{\text{cov}~}
        for $K_g$ sufficiently large, i.e. $K_g\geq N+1$.\footnote{This elementary  result is due to Blichfeldt \cite{blichfeldt1921notes}. We thank David Conlon and Huy Pham for pointing this to us.} Thus, we have
            \[\mathbb P[K_g\geq K]\lesssim\sum_{1\leq M \leq N} \sum_{\text{dyadic }l\in(1, \infty)} \sum_{\mathcal{P}: \text{cov}(\mathcal P)\sim l} \mathbb P[|\spn_\R(\mathcal P)\cap gT|\gtrsim lK_0]
                \label{eqn-sum-over-dyadics}\]
        Let $\delta>0$ be a small dyadic number and let $\mathcal P_\delta$ be the dyadic box equivalent to $B_{\delta\inv}\cap\mathcal N_\delta(\spn_\R(\mathcal P))$, and denote by $s=(s_1,\cdots,s_N)$ the dimensions of $\mathcal P_\delta$. Assume that $\delta$ is sufficiently small, so that
            \[|\spn_\R(\mathcal P)\cap gT|\sim\delta\ie M|\mathcal P_\delta\cap gT|\]
        for all possible values of $g$. Note that $s\meet t=(t_1, \ldots, t_M, \delta,\cdots,\delta)$, so \cref{lemma-intersection-estimate} implies that
            \[\mathbb P[|\spn_\R(\mathcal P)\cap gT|\gtrsim lK_0]\lesssim\p{lK_0\prod_{j=1}^M t_j\inv}^{-N+0.99}I(s\meet t,s\join t)\]
        Note that
            \[I(s\meet t,s\join t)=\prod_{i=1}^M\prod_{j=M+1}^Nt_i\inv t_j=|T|^M\p{\prod_{j=1}^M t_j\inv}^N\]
        Therefore, it follows that
            \[\mathbb P[|\spn_\R(\mathcal P)\cap gT|\gtrsim lK]\lesssim(lK)\ie N\lambda^{0.99}
                \label{eqn-subspace-vol}\]
        where $\lambda=lK\prod_{j=1}^M t_j\inv$. Note that we need only consider the case when $\lambda\lesssim1$, since the largest possible intersection of $T$ with an $M$-dimensional subspace has volume $\sim t_1 \cdots t_M$.%$\mathbb P[|\spn_\R(\mathcal P)\cap gT|\gtrsim lK]=0$ if we do not have $\lambda\lesssim1$.

        Next, we would like to compute the number of possible $M$-dimensional $\mathcal P$ with $\cov(\mathcal P)\sim l$. Let $\alpha(l,M)$ denote the number of $M$-dimensional $\mathcal P$ with $\cov(\mathcal P)\sim l$. We will show by induction on $M$ that $\alpha(l,M)\lesssim l^N$. For each $M$-dimensional $\mathcal P$, let $\mathcal P'\subset\mathcal P$ denote the codimension-$1$ sublattice of $\mathcal P$ with minimum covolume, say $\cov(\mathcal P')\sim l'$. Note that, since $\mathcal P$ admits an almost-orthogonal set of generators, we have $l'\lesssim l\ef{M-1}M$.
        
        Since $\mathcal P$ is determined uniquely by $\mathcal P'$ and a generator of $\mathcal P/\mathcal P'\subset\Z^N/\mathcal P'$, we have
            \[\alpha(l,M)\lesssim\sum_{\text{dyadic }l'\in(1,l\ef{M-1}M)}\alpha(l',M-1)l^{N-M+1}(l')^{M-N}\]
        To see where the factor $l^{N-M+1}(l')^{M-N}$ comes from, note that any representative $x$ of the coset generating $\mathcal P/\mathcal P'$ must lie a distance at most $\frac l{l'}$ away from $\spn_\R(\mathcal P')$. Furthermore, the $x$ may be chosen to project into the fundamental domain $F'$ of $\mathcal P'$, i.e. $\pi_{\spn_\R(\mathcal P')}(x)\in F'$. Thus, $x$ lies in a box with volume $\p{\frac l{l'}}^{N-M+1}l'=l^{N-M+1}(l')^{M-N}$  since any possible generator of $\mathcal P/\mathcal P'$ must lie in $\Z^N\cap\mathcal N_{l/l'}(\spn_\R(\mathcal P'))$, which leaves $\lesssim\p{\frac l{l'}}^{N-M+1}(l')=l^{N-M+1}(l')^{M-N}$ choices. By our induction hypothesis, we conclude
            \[\alpha(l,M)\lesssim\sum_{\text{dyadic }l'\in(1,l\ef{M-1}M)}(l')^Nl^{N-M+1}(l')^{M-N}=\sum_{\text{dyadic }l'\in(1,l\ef{M-1}M)}l^N\cdot\frac{(l')^M}{l^{M-1}}\lesssim l^N\]
        Combining this with \cref{eqn-subspace-vol}, we deduce
            \[\sum_{\text{dyadic }l\in(1, \infty)} \sum_{\mathcal{P}: \text{cov}(\mathcal P)\sim l} \mathbb P[|\spn_\R(\mathcal P)\cap gT|\gtrsim lK_0]\lesssim\sum_{\text{dyadic }\lambda\in(0,1)}K_0\ie N\lambda^{0.99}\lesssim K_0\ie N\]
        This, along with \cref{eqn-sum-over-dyadics}, completes the proof.        
\end{proof}

\begin{remark}
    \cref{thm-convex-density} gives an interesting example of the Kakeya maximal function. Recall that for a positive function $F$ on $\R^N$, a $\delta \in (0, 1)$ and a direction $v \in S^{N-1}$, the Kakeya  maximal function  \[\mathcal{K}_{\delta} F (v) = \sup_{T: \delta \times \cdots \times \delta \times 1-\text{tube in direction }v} \frac{1}{|T|}\int_T F.\] The Kakeya Maximal Conjecture \cite{bourga1991besicovitch} predicts \[\|\mathcal{K}_{\delta} F\|_{L^N (\mathbb S^{N-1})} \lessapprox_N \|F\|_{L^N (\R^N)}.\] This conjecture is known in dimension $2$ but is widely open in all higher dimensions. Recently Wang-Zahl \cite{kakeya-3d} made a  breakthrough towards this conjecture by proving the related \emph{Kakeya set conjecture} in dimension $3$. We refer the readers to \cite{kakeya-3d} for history and  recent progress.

    Since \cref{thm-convex-density} is also a natural $L^N$ estimate, we are curious to see if it is related to a Kakeya maximal estimate. Let us take $F_0 = 1_X$ where $X$ is a union of $\delta$-balls around points in $\delta^{\frac{N-1}{N}}\Z^N$ inside the unit ball. Then the statistics of $\sup_{T: \delta \times \cdots \times \delta \times 1-\text{tube in direction }v}\int_T F_0$ reduces to counting the intersection between $\delta \times \cdots \times \delta \times 1$ tubes with $\delta^{\frac{N-1}{N}}\Z^N$. By rescaling, this corresponds to estimating the intersection between $\delta^{\frac{1}{N}} \times \cdots \times \delta^{\frac{1}{N}} \times \delta^{-\frac{N-1}{N}}$-tubes with $\Z^N$. These tubes have volume $\sim 1$ and we can use \cref{thm-convex-density}. If $T$ is a fixed $\delta^{\frac{1}{N}} \times \cdots \times \delta^{\frac{1}{N}} \times \delta^{-\frac{N-1}{N}}$-tube, then \cref{thm-convex-density} gives \[\|\mathcal{K}_{\delta} F_0\|_{L^N (\mathbb S^{N-1})} \sim \|\delta I_T\|_{L^N (SO(N))}\lessapprox \delta.\]
    On the other hand, \[\|F_0\|_{L^N (\R^N)}\approx \delta^{\frac{1}{N}}.\]
    In conclusion, \cref{thm-convex-density} is an $L^N$ Kakeya maximal estimate for the lattice-based function $F_0$ in any dimension $N$.  This estimate does not follow from the $L^N$ Kakeya maximal conjecture in $\R^N$ and is much stronger. This is because we are looking at a very special and interesting $F_0$.
\end{remark}
    \def\GHM{\mathrm{GHM}}
    \def\smallscale{\rho}

    \begin{corollary}\label{cor-convex-density}
        Let $R>1$, and let $V\subset\R^N$ be a uniformly distributed random $n$-dimensional subspace of $\R^N$. For any $\kappa\in(0,1)$, we have with probability at least $\kappa$ the uniform estimate
            \[\frac{|S\cap \Z^N|}{|S|}\lesssim_{N, \kappa} R^{O\p{\frac{n^2}N}}.
                \label{eqn-cor-density-physical}\]
        where $B_{R\inv}^N\subsetsim S\subsetsim B_R^N$ is a cylinder of any height over any convex subset in $V$ with $|S|\geq 1$.
    \end{corollary}
        \begin{proof}
        Suppose that $V=g\R^n$ where $g\in SO(N)$ is a uniformly distributed random variable. We assume that $S$ is equivalent to a box $gh_0[s]$, for some $h_0\in SO(n)$ and some $s=(s_1,\cdots,s_N)$. This requires that $S$ is centered at the origin, which we may assume since, if $S$ is not centered at the origin, then the sumset $S-S$ is centered at the origin and contains at least as many lattice points as did $S$.

        Furthermore, we may assume without loss of generality that $|s|=1$, since if $|s|>1$, then we may tile $S$ with convex sets of unit volume; the density of $\Z^N$ in $S$ cannot exceed the density of $\Z^N$ on every tile, and hence the result when $|s|>1$ follows from the result when $|s|=1$.
        
        Next, we fix $K>0$, and let $\chi_g(s)$ denote the event that $|\Z^N\cap gb[s]|\geq K$ for some $b\in SO(n)$. It suffices to show that $\mathbb P(\chi_g(s))<1-\kappa$ for some $K\sim_{N, \kappa}R^{O\p{\frac{n^2}N}}$.

        If $h\in SO(n)$ is a uniformly distributed random variable independent of $g$ and $b$, note that
            \[\mathbb P[|\Z^N\cap ghS|\gtrsim K\mid\chi_g(s)]\gtrsim\mathbb P[|gh_0S\cap ghS|\sim|S|]\sim I(s',s')\]
        by \cref{lemma-dyadic-contain-prob}, if $s'$ denotes the dimensions of $S'=S\cap V$. On the other hand, by \cref{thm-convex-density} we have
            \[\mathbb P[\chi_g(s)]\mathbb P[|\Z^N\cap ghS|\gtrsim K\mid\chi_g(s)]\lesssim K\ie N.\]
        Therefore, we have
            \[\mathbb P[\chi_g(s)]\lesssim K\ie N(I(s',s'))\inv.\]    
        This implies that it suffices to show that
            \[\sum_sI(s',s')\inv\lesssim R^{O(n^2)}\]
        where the sum of over all $s'$ with $(R\inv,\cdots,R\inv)\lesssim s'\lesssim(R,\cdots,R)$. This is a crude estimate and we use a crude bound on $I(s',s')$, i.e. $I(s',s')\gtrsim R^{-2n^2}$ since each of the $n^2$ factors in the definition of $I(s',s')$ is at least $\frac{R\inv}R$.
       \end{proof}
\section{The construction of the counterexample}\label{counterexamplesection}
\subsection{Some preliminary constructions}
    \def\smallscale{\rho}
    With the above preparation, we prove \cref{powerblowupthm} in this section. We will prove a generalization that also fits in the setting of \cref{subsec-codim}.
    
    Let $n\geq2$ be a dimension and  $k\geq 2$ be an integer. We will prove the following:
    \begin{theorem}\label{thm-counterexample}
         Let $1\leq m<n$ be integer dimensions, let $l\in[1,n-1]$ be another integer dimension, and let $k\geq2$. Let $\Sigma$ be any $m$-dimensional compact $C^k$ submanifold of $\R^n$, parameterized by some $s:[0,1]^m\to\R^n$. Then, for any $\smallscale>0$, there exists a perturbation $\Sigma'$ of $\Sigma$, given as the graph of some $s':[0,1]^m\to\R^n$ satisfying $\pl D^\gamma(s-s')__\infty\leq\smallscale$ for all indices $|\gamma|<k$, that satisfies
            \[\p{C_{\Sigma,l}'(R)}^2\geq  R^{\alpha-\eps}\]
         for every $\eps>0$ and every  $R\geq O_\eps(1)$, where $\alpha=\frac{lm}{m+k(n-m)}$.
    \end{theorem}
    \begin{corollary}
        There exists a strictly convex, $C^2$ hypersurface $\Sigma\subset\R^n$,  such that for every $\eps > 0$, for all sufficiently large $R>0$ one can find $f\in L^2(\Sigma)$ and $w:\R^n\to[0,1]$ supported in $B_R$ such that
            \[\int_{B_R}|\ext f(x)|^2w(x)dx\geq  R^{\frac{n-1}{n+1}-\eps}\pl f__2{\Sigma}^2\pl\mathcal M(w)__\infty\]
    \end{corollary}
    Notably, $\alpha=\frac{n-1}{n+1}$ matches the exponent in \cite{ciw-mt-24}. Thus, the constant in the MT conjecture for $C^2$ hypersurfaces must be $R\ef{n-1}{n+1}$, up to the endpoint.

    To lessen the technical load, reduce the result to the following proposition, allowing us to fix $\eps,R$:
    \begin{proposition}\label{prop-counterexample}
        For any $s_0:[0,1]^m\to\R^n$ parameterization of any $m$-dimensional $C_k$ submanifold $\Sigma_0$ with $\pl D^\gamma(s_0)__\infty\lesssim1$ for every index $|\gamma|\leq k$, any $\eps > 0$, and any number $R\gtrsim1$, there exists:
        \begin{enumerate}
            \item Some $t_R:\R^m \to\R^n$ supported on $[0, 1]^m$ with $\pl t_R__\infty\lesssim1$ and $\pl D^\gamma(t_R)__\infty\leq1$ for all indices $|\gamma|\leq k$.
            \item Some weight $w_R:B_R^n\to[0,1]$ such that
                \[\pl w_R__1{B_R}\lesssim R^{n-l},\]  \[ \sup_{U\subset\R^n:\dim U=l}\int_Uw_Rd\lambda_l\lesssim_\eps R^{\eps}\]
            and
                \[\pl\hat w_R__2{\Sigma_R;d\sigma}^2\gtrsim R^{-\eps} R^{n-m-l}R^\alpha\pl w_R__1,\]
            where $\Sigma_R$ denotes the graph of the function $s_0+t_R$, $d\sigma$ is the corresponding surface measure, and $\alpha=\frac{lm}{m+k(n-m)}$.
        \end{enumerate}
    \end{proposition}
    Intuitively, \cref{prop-counterexample} states that, given any $C^k$ hypersurface $\Sigma_0$, there is some surface $\Sigma'$ within $R\ief k{m+k(n-m)}$ of $\Sigma_0$ so that $C_{\Sigma,l}(R)\gtrapprox R\ef{lm}{m+k(n-m)}$.
    
    \begin{lemma}\label{lemma-technical}
        \cref{prop-counterexample} implies \cref{thm-counterexample}.
    \end{lemma}
    \begin{proof}
        Consider a sequence of scales $R_j=2^{2^{j^{0.99}}} (j \geq 1)$. We construct the perturbation of $\Sigma$ based on these scales. For all $j$, we can find a patch on $\Sigma$ with side length $\leq \frac{1}{10} 4^{-j}$ and can make all patches found this way disjoint. We perturb $\Sigma$ on each patch individually. That is, for each patch, we first rescale it to a graph of $[0, 1]^m$ and use \cref{prop-counterexample} to obtain a perturbation and rescale back (since $\frac{1}{10} 4^{-j} = R_j^{-o(1)}$, we can afford to lose $R_j^{o(1)}$ along the way of rescaling). Therefore, we have
            \[{C_{\Sigma,l}'(R_j)}^2\gtrapprox  R_j^{\alpha}.\label{gtrapproxforRj}\]
        
        Next, note we have the slow growth condition $\log_{R_j}(R_{j+1})=2^{(j+1)^{0.99}-j^{0.99}}\to 1$ as $j\to\infty$, hence \cref{gtrapproxforRj} can be extended to all scales $R>1$.

        Finally, note that the rescaling implies that the perturbation go to zero in the $C^k$ topology, so we can ensure that $\Sigma'$ is arbitrarily close to $\Sigma$ by ignoring the first several scales $R_j$.
    \end{proof}

    Finally we prove \cref{prop-counterexample}.
    \subsection{The proof of \cref{prop-counterexample}}
    
         Let $N\gg n$ be a sufficiently large dimension, depending on $\eps$, and let $\kappa\in(0,1)$ be fixed (can be chosen to be $99\%$ for example).  Let $\mathcal L=R\ief lN\Z^N$ and let $\mathcal L^*$ be the dual lattice $R\ef lN\Z^N$. Next, define a weight $W_R(x):\R^N\to[0,\infty)$ by
            \[W_R(x)=\Big|(\dd_1*\delta_{\mathcal{L}^*})\bb_R\Big|^2
                \label{def-W-R}\]
     Let $V\subset\R^N$ be a random $n$-dimensional subspace of $\R^N$, and let $\iota:\R^n\to V$ be a random linear isometry multiplied by a uniformly distributed random variable in $[1,2]$.\footnote{The random rescaling here is necessary, as this ensures the transforms of $\Sigma$ to hit many points with uniform probability, which ultimately guarantees \cref{geomintersection} below. Without the rescaling, \cref{geomintersection} will not follow from elementary geometry when $\Sigma$ is, e.g., a sphere centered at the origin.}
     Let
            \[w_R:=W_R\circ\iota
                \label{def-w-R}\]
        be a weight on $\R^n$. That is, $W_R$ looks like a rescaled integer lattice in $\R^N$ and $w_R$ is the restriction to a random $n$-dimensional subspace. Note that $W_R=|\hat H_R|^2$, where $H_R$ is given as
            \[H_R(x)=R\ie l\dd_{R\inv}*\bb_1\delta_{\mathcal L}
                \label{def-H-R}\]
        Similarly, $w_R=|\hat h_R|^2$, where $\hat h_R:=\hat H_R\circ\iota$.
        
        We pause to control $\pl w_R__1$. Note that $\pl w_R__1\sim|\mathcal N_1(V)\cap\mathcal L^*\cap B_R|$ by the definition of $W_R$. Therefore, $\pl w_R__1\sim|\Z^N\cap V'|$, where $V'$ is the convex set $\mathcal N_{R\ief lN}(V\cap B_R)$ with $|V'|\sim R^{n-l}$. Recall that $V$ is a random variable whose orientation is uniformly distributed in $SO(N)$, so we may apply \cref{thm-convex-density} to prove a probabilistic upper bound on $|\Z^N\cap V'|$, i.e.
            \[|\Z^N\cap V'|\lesssim_{\kappa'}|V'|\]
        with probability at least $\kappa'$.  We will use some more probabilistic arguments of this flavor later on, and can  set now $\kappa'=0.9+0.1\kappa>\kappa$, to give extra room for error.  Furthermore, note that  $V'$ must contain $\gtrsim |V'|$ points in $\Z^N$. To see why, note that a random translation of the set $\frac12V'$ contains an average of $2\ie N|V'|$ points in $\Z^N$; in particular, some translation of $x_0+\frac12V'$ contains a set $L_0$ with $|L_0|\geq 2\ie N|V'|$, so we have $L_0-L_0\subset\frac12V'-\frac12V'=V$ by the convexity of $V$. In summary, we have proved 
        
        %the convex body $R^{-\frac{l}{N}} (N_1(V_0)\cap B_R)$  where $V_0$ is an arbitrary linear subspace of dimension $n$. This convex body has volume $\sim R^{n-l}$ and, by \cref{thm-convex-density} applied to a convex body of volume $1$ that is similar to $R^{-\frac{l}{N}} (N_1(V_0)\cap B_R)$, we see that for any probability $\kappa'\in(0,1)$, we have 
            \[\pl w_R__1\sim_{\kappa'} R^{n-l}
               \label{eqn-l1-norm-w-R}\]
        with probability at least $\kappa'$.

               %have \redit
%            \[\mathbb E[\pl w_R__1]\sim \sum_{\text{dyadic } r \in [1, R]}\frac{|\mathcal L^*\cap A_r^N||\mathcal N_1(V)\cap A_r^N|}{|A_r^N|}\sim R^{n-l}
 %               \label{eqn-expected-l1-norm-w-R}\]
%where $A_r^N$ is the annulus $\{r< |x|< 2|r|\}$ in $\R^N$.\rend
                
%        \bedit On the other hand, it is possible to tile $B_R^N$ with $\sim R^{N-n}$ translations of the box $\mathcal N_{\frac12}(V)\cap B_{\frac R2}$. Since there are a total of $\sim R^{N-l}$ points in $\mathcal L^*\cap B_R$, we deduce that there is at least one such region $W$ containing $\gtrsim R^{n-l}$ points in $\mathcal L^*\cap B_R$. In particular, $\gtrsim R^{n-l}$ points in $\mathcal L^*$ also lie in the difference set $W-W\subset\mathcal N_1(V)$. This implies that $\pl w_R__1{B_R}\gtrsim R^{n-l}$. Combining this with \cref{eqn-expected-l1-norm-w-R}, we see from elementary first moment estimates that the random variable $\pl w_R__1{B_R}$ is bounded with arbitrary high probability. More quantitatively, for any probability $\kappa'\in(0,1)$, we have 
%            \[\pl w_R__1\sim_{\kappa'} R^{n-l}
%                \label{eqn-l1-norm-w-R}\]
%        Since we will see more probabilistic %assertions of this flavor later on, we set $\kappa'=0.9+0.1\kappa>\kappa$, to give extra room for error.
%        \bend
        
        Next, note that by the projection-slicing theorem we have
            \[h_R=R\ie l\dd_{R\inv}*\sum_{v\in\mathcal L}\bb_1(\iota^t v)\delta_{\iota^t (v)}\]
        where $\iota^t$ denotes the transpose of $\iota$.\footnote{Strictly speaking, so far we have two $\phi_{R\inv}$ in dimensions $N$ and $n$. We slightly abuse the notation by taking the measure $\phi_{R\inv} dx$ in dimension $n$ to be the pushforward of  the measure $\phi_{R\inv} dx$ in dimension $N$.} Next, note that
            \[\hat w_R=h_R*\tilde h_R=R^{-2l}\dd_{R\inv}*\dd_{R\inv}*\sum_{v_1,v_2\in\mathcal L}\bb_1(\iota^t  v_1)\bb_1(\iota^t v_2)\delta_{\iota^t(v_1-v_2)}\gtrsim |\mathcal L_0|R^{-2l+n}1_{\mathcal N_{R\inv}(\iota^t(\mathcal L_0))},
                \label{rapid-decay-difference}\]
        where $\mathcal L_0=\mathcal L\cap B_1$, since for any $v\in \mathcal L_0$, there are $\sim|\mathcal L_0|$ solutions to $v_1-v_2=v$. In \cref{subsec-construct-p}, we will construct $t_R$ so that $|\iota^t(\mathcal L_0)\cap\Sigma_R|\gtrapprox R^\alpha$. In this case, for any $\xi\in\mathcal N_{R\inv}(\Sigma_R\cap\iota^t(\mathcal L_0))$, we have
            \[\hat w_R(\xi)\gtrsim R^{-2l+n}|\mathcal L_0|\sim R^{n-l}.\]
            
        Thus,
            \[d\sigma\Big(\hat w_R\inv\big([R^{n-l},\infty)\big)\Big)\gtrapprox R^\alpha R\ie m.\]
            
        Hence, $\|\hat w_R\|_{L^2 (\Sigma)}^2\gtrapprox R^\alpha R\ie mR^{2n-2l}\sim R^\alpha R^{n-l-m}\pl w_R__1$.

        Note that \cref{cor-convex-density} implies that, with probability at least $\kappa'$, we have 
            \[\sup_U\int_Uw_Rd\lambda_l\lesssim_\kappa R^{O_n(N\inv)},\]
        where $U$ is allowed to be any $l$-dimensional affine subspace of $\R^n$. To see why, note that the left-hand side is $\sim |\mathcal L^*\cap\mathcal N_1(\iota(U))| \sim|\Z^N\cap \iota(U')|$, where $U'=R\ief lN(B_R\cap\mathcal N_1(U))$. Now $U'$ lies in a cylinder of height $R\ief lN$ over a convex set contained in $V$ with $|U'|\sim R\ie l|U|\sim 1$. We can apply \cref{cor-convex-density} to deduce
            \[\pl \mathcal M_l(w)__\infty\lesssim R^{O_n(N\inv)}\lesssim R^{\eps}\]
        for $N$ sufficiently large.
        It remains to show that $\Sigma_R$ can be chosen so that $|\iota^t(\mathcal L_0)\cap\Sigma_R|\gtrapprox R^\alpha$ with probability at least $\frac{1+\kappa}2$.

    \subsection{The construction of  $\Sigma_R$.}\label{subsec-construct-p}
    Now, we would like to show that we may construct $\Sigma_R$ (abbreviated as $\Sigma$ henceforth) so that $|\iota^t(\mathcal L_0)\cap\Sigma|\gtrapprox R^\alpha$, where $\alpha=\frac{lm}{m+k(n-m)}$.

    Recall that we are given a starting $m$-manifold $\Sigma_0$, parameterized by a function $s_0:[0,1]^m\to\R^n$, and we would like to find some function $s':[0,1]^m\to\R^n$ so that $\Sigma$ such that the image of $s=s_0+s'$ contains $\gtrapprox R^\alpha$ points in $\iota^t(\mathcal L_0)$.  Furthermore, we must choose $s'$ so that all its derivatives up to order $k$ are bounded.

    First, we partition $[0,1]^m$ into a set $\{U_\theta:\theta\in\Theta\}$ of $R^{m\beta}$ many cubes, each of side length $R\ie\beta$, for some $\beta\leq\frac12$ that will be chosen later. For each $\theta\in\Theta$, $s_0(U_\theta)$ is contained in a slab $S_\theta$ of dimensions $(R\ie\beta)^{\times m}\times(R\inv)^{\times(n-m)}$.

    On each square $U_\theta$, we would like to perturb $s_0$ by an amount $a_\theta$. That is, let $s_\theta=a_\theta t_\theta$, where $t_\theta=\bb_{R\ie\beta}\circ\tau_\theta$ is a bump function on $U_\theta$ and $\tau_\theta$ denotes the translation from $[0,R\ie\beta]$ to $U_\theta$. We then choose $s'=\sum_{\theta\in\Theta}s_\theta$.

    Since $s_\theta$, and all of its derivatives have support $\subsetsim U_\theta$, it suffices to show that its derivatives up to order $k$ are all bounded. For this, note that if $|\gamma|=j$, then $\frac{\partial^j\bb_{R\ie\beta}}{\partial\xi_1^{\gamma_1}\cdots\partial\xi_{n-1}^{\gamma_{n-1}}}\lesssim R^{\beta j}$. Therefore, we require $a_\theta\lesssim R\ie{\beta k}$, in order for the derivatives of $s_{\theta}$ to remain bounded.

    Let $S_\theta'$ denote the $R\ie{\beta k}$-thickening of $S_\theta=s_0(U_\theta)$. For each $\theta\in\Theta$, we say that $\theta$ is \textit{good} if $S_\theta'$ contains at least one point in $\mathcal L$. Note that $|S_\theta'|\sim R^{-m\beta-(n-m)\beta k}=R^{-\beta(m+(n-m)k)}$, and $|\iota^t(\mathcal L_0)|\sim R^l$, so if we choose $\beta\geq\frac l{m+(n-m)k}=\frac\alpha m$, then
        \[\mathbb E|S_\theta\cap\iota^t(\mathcal L_0)|\gtrsim1\]
    By \cref{cor-convex-density}, if we set $\beta=\frac\alpha n$, then with probability at least $\frac{1+\kappa}2$, we claim that there are $\gtrsim_\kappa R^\alpha R\ie\eps$ many good $\theta\in\Theta$. To see why, note that
        \[|S_\theta\cap\iota^t(\mathcal L_0)|=|B_R\cap(\iota^t)\inv(S_\theta)\cap R\ief lN\Z^N|=\abs{B_{R\ef{l+N}N}\cap R\ef lN(\iota^t)\inv(S_\theta)\cap\Z^N}\]
    and $B_{R\ef{l+N}N}\cap R\ef lN(\iota^t)\inv(S_\theta)$ satisfies the hypotheses of \cref{cor-convex-density}, so we have with probability at least $\frac{1+\kappa}2$ that $|S_\theta\cap\iota^t(\mathcal L_0)|\lesssim R^\eps$, since $N$ is sufficiently large. On the other hand, for a fixed $\theta$, we have by elementary geometry considerations, \[\mathbb E[|S_\theta\cap\iota^t(\mathcal L_0)|]\sim|S_\theta||\iota^t(\mathcal L_0)|\sim 1.\label{geomintersection}\] 
    
    Therefore,
        \[\mathbb E\b{\abs{\bigcap_{\theta\in\Theta}S_\theta}\cap\iota^t(\mathcal L_0)}\sim|\Theta|\sim R^\alpha\]
    This means that the number of good theta is $\gtrsim_{\kappa,\eps}R^{\alpha-\eps}$. This completes the proof.

\section{$C^k$ hypersurfaces through many lattice points}\label{richsurfacessection}
Using \cref{thm-convex-density} and the idea in \cref{counterexamplesection}, we now rigorously present the way to construct hypersurfaces with many lattice points as mentioned in \cref{Latticediscussionsec}.

\begin{theorem}\label{sharplatticethm}
    For any dimension $n\geq2$, any $C^2$ hypersurface $\Sigma$,  any $\eps>0$ and any  scale $R\gtrsim_{\Sigma, \eps} 1$, there is an $\eps$-small $C^2$-perturbation $\Sigma_R '$ of $\Sigma$  that contains $\gtrsim_{\Sigma, \eps} R^{n-2+\frac2{n+1}}$ points in the lattice $\p{\frac1R\Z}^n$.
    
    %Let $n\geq2$ be any dimension, let $\Sigma\subset\R^n$ be any compact $C^2$ hypersurface, and let $U$ be an arbitrarily small neighborhood of $\Sigma$ in the $C^2$. Then, for $M>1$ sufficiently large, there is a $C^2$ perturbation of $\Sigma$ that contains $\sim M^{n-2+\frac2{n+1}}$ points in $\p{\frac1M\Z}^n$ that satisfies
    %For any given $C^2$ compact hypersurface $\Sigma \subset \R^n$ and any $M>1$, there is an arbitrarily small $C^2$ perturbation of $\Sigma$ that contains $\sim M^{n-2 + \frac{2}{n+1}}$ points in $\p{\frac{1}{M}\Bbb{Z}}^n$.
\end{theorem}
%\begin{proof}
%    Let $\Sigma$ be any compact $C^2$ hypersurface, and let $\eps>0$. We would like to construct a $C^2$ hypersurface $\Sigma'$, so that $\Sigma'$ contains $\gtrsim R^{n-2+\frac1{n+1}}$ points in the lattice $\p{\frac1R\Z}^n$, for any dyadic $R\gg0$ sufficiently large.
    
%    We will construct $\Sigma'$ by an iterative process as follows
%\end{proof}
\begin{proof}
    First, divide $\Sigma$ into $R^{n-2+\frac2{n+1}}=R\ef{n(n-1)}{n+1}$ many caps $S_\theta$, each of diameter $\sim R\ief n{n+1}$, and let $S_\theta$ be the slab with dimensions $\sim R\ief n{n+1}\times\cdots\times R\ief n{n+1}\times R\ief{2n}{n+1}$ that contains $S_\theta$. Next, note that each of the boxes $S_\theta$ are equivalent up to a rotation; furthermore, we may assume without loss of generality that the unit normal vectors to each of the slabs $S_\theta$ are separated by a distance $\gtrsim_\eps R\ief n{n+1}$; this is possible by replacing $\Sigma$ with some $C^2$ perturbation whose Gauss map is an immersion.
    
    Let $S_\theta'$ denote a translation of $S_\theta$ that is centered at the origin. Let us fix $K$ and say that $S_\theta'$ is \textit{good} if it contains $\leq K$ points in $\p{\frac1R\Z}^n$. By \cref{thm-convex-density}, we conclude that for some $K\lesssim_n1$, at least $\frac12$ of the total number of $\theta$ are good. Let $\mathcal S$ be the union of all good $S_\theta$. Assume without loss of generality that $\mathcal S$ contains $\gtrsim|\mathcal S|\sim R^{n-2+\frac2{n+1}}$ many points in $\p{\frac1R\Z}^n$; this is possible by requiring $\eps\gtrsim R\inv$ and replacing $\Sigma$ with a translated copy $\Sigma+O(\eps)$.

    Finally, we conclude that there are $\gtrsim R^{n-2+\frac2{n+1}}$ many caps $S_\theta$ which contain at least $1$ point in $\p{\frac1R\Z}^n$, and so we complete the proof by perturbing $\Sigma$ locally on each such cap as in the proof of \cref{prop-counterexample}.

\end{proof}

\begin{remark}
    %As mentioned in \cite{iosevich2001curvature}, a classical result of Landau \redit(see also \cite{andrews1961asymptotic}) \rend asserts that 
    A classical result of Andrews \cite{andrews1961asymptotic} asserts any strictly convex $C^2$ compact hypersurface $\Sigma \subset \R^n$ must contain $\lesssim M^{n-2 + \frac{2}{n+1}}$ points in $(\frac{1}{M}\Z)^n$. This bound can be recovered by the recent decoupling approach of Kiyohara \cite{kiyohara2024lattice}. Our construction in \cref{sharplatticethm} matches this bound and shows it is sharp in all dimensions.
    
    In dimension $2$, Jarn\'{i}k already constructed a family of strictly convex $C^2$ curves that have $\sim M^{\frac{2}{3}}$ points in $(\frac{1}{M}\Bbb{Z})^2$ (see e.g. \cite{iosevich2001curvature} for an introduction). As a comparison, our construction seems more flexible in the following sense: In Jarn\'{i}k's construction, the curve is always a perturbation of a parabola \cite{martin2003limiting}, but our construction can be a perturbation of any strictly convex $C^2$ curve. Moreover, this method of construction can also produce points on $C^k$ submanifolds of arbitrary dimension and codimension, see \cref{Ckcounting} below.

     Jarn\'{i}k's construction was refined by Plagne \cite{plagne1999uniform} (see also \cite{iosevich2007mean}), who constructed a fixed strictly convex $C^2$ curve in $R^2$ that contains $\approx M^{\frac{2}{3}}$ points in $\p{\frac{1}{M}\Z}^2$ for a sequence of $M$’s going to infinity. By imitating the proof of \cref{lemma-technical}, we can obtain a similar refinement in all higher dimensions (\cref{latticethm} below). %and \cref{latticecor} below).
     
     \end{remark}
     %And if one is willing to lose an arbitrarily small power of $M$, again by following the proof of \cref{lemma-technical}, we can have one hypersurface with the lower bound of points valid for \emph{every scale} $M>1$, not just a sequence going to $\infty$.
\begin{theorem}[Generalization of Theorem 1 in \cite{plagne1999uniform}]\label{latticethm}
   Let $\chi: \mathbb{Z}^+ \to \mathbb{R}^+$ be any function tending to infinity at $\infty$.  For any dimension $n\geq2$, any $C^2$ hypersurface $\Sigma$  and any $\eps>0$, there is an $\eps$-small $C^2$-perturbation $\Sigma '$  of $\Sigma$ and a strictly increasing  positive integer sequence $\{q_M\}_{M\geq 1}$ (depending on $n, \eps, \Sigma$), such that\[\abs{\Sigma' \cap \p{\frac{1}{q_M}\mathbb{Z}}^n}\gtrsim \frac{q_M^{n-2+\frac{2}{n+1}}}{\chi (q_M)}.\]
    \end{theorem}

     %\begin{theorem}\label{latticecor}
  %For any dimension $n\geq2$, any $C^2$ hypersurface $\Sigma$  and any $\eps>0$, there is an $\eps$-small $C^2$-perturbation of $\Sigma$ that contains $\gtrapprox_{\Sigma, \eps} R^{n-2+\frac2{n+1}}$ points in the lattice $\p{\frac1R\Z}^n$ for every dyadic $R \gtrsim_{\Sigma, \eps} 1$.\rend
   % \end{theorem}
    %Let $n\geq2$ be any dimension, let $\Sigma\subset\R^n$ be any compact $C^2$ hypersurface, and let $U$ be an arbitrarily small neighborhood of $\Sigma$ in the $C^2$. Then, for $M>1$ sufficiently large, there is a $C^2$ perturbation of $\Sigma$ that contains $\sim M^{n-2+\frac2{n+1}}$ points in $\p{\frac1M\Z}^n$ that satisfies
    %For any given $C^2$ compact hypersurface $\Sigma \subset \R^n$ and any $M>1$, there is an arbitrarily small $C^2$ perturbation of $\Sigma$ that contains $\sim M^{n-2 + \frac{2}{n+1}}$ points in $\p{\frac{1}{M}\Bbb{Z}}^n$.

By repeating the proof of \cref{sharplatticethm} and \cref{latticethm}, we have the following generalizations to $C^k$ submanifolds:

\begin{theorem}[Generalization of \cref{sharplatticethm}]\label{Ckcounting}
    For any dimensions $n\geq2, m<n$, any $C^k (k \geq 2)$ compact $m$-dimensional submanifold $\Sigma\subset \Bbb{R}^n$,  any $\eps>0$ and any  scale $R \gtrsim_{\Sigma, \eps} 1$, there is an $\eps$-small $C^k$-perturbation $\Sigma_R '$ of $\Sigma$  that contains $\gtrsim_{\Sigma, \eps} R^{\frac{mn}{m+k(n-m)}}$ points in the lattice $\p{\frac1R\Z}^n$.
\end{theorem}

\begin{theorem}[Generalization of \cref{latticethm}]\label{Cklatticethm}
    Let $\chi: \mathbb{Z}^+ \to \mathbb{R}^+$ be any function tending to infinity at $\infty$.  For any dimensions $n\geq2, m<n$, any $C^k (k \geq 2)$ compact $m$-dimensional submanifold $\Sigma\subset \Bbb{R}^n$,  any $\eps>0$ and any  scale $R \gtrsim_{\Sigma, \eps} 1$, there is an $\eps$-small $C^k$-perturbation $\Sigma '$ of $\Sigma$ and a strictly increasing  positive integer sequence $\{q_M\}_{M\geq 1}$ (depending on $k, n, \eps, \Sigma$), such that\[|\Sigma' \cap (\frac{1}{q_M}\mathbb{Z})^n|\gtrsim \frac{q_M^{\frac{mn}{m+k(n-m)}}}{\chi (q_M)}.\]
    \end{theorem}

In terms of smooth hypersurfaces, Schmidt \cite{schmidt1986integer} has the following conjecture:

\begin{conjecture}\label{Schmidtconj}
    Let $n\geq3$ be a dimension and let $\Sigma\subset\R^n$ be a compact $C^\infty$ hypersurface. Then, for any $R>1$, $\Sigma$ contains at most $\sim_\Sigma R^{n-2}$ points in $\p{\frac1R\Z}^n$.
\end{conjecture}

Our \cref{Cklatticethm} implies that if one wishes to prove \cref{Schmidtconj} by working only on $C^k$ hypersurfaces, then it is necessary for them to consider some $k\geq 4$ in dimension $3$, or $k\geq 3$ in dimensions $4$ and higher. 

For further references in the approximation theory of rescaled submanifolds by integer points, we refer the reader to \cite{iosevich2011lattice} and references therein.

\printbibliography
\end{document}